\documentclass{amsart}
\usepackage{amsfonts}
\usepackage{graphicx}
\usepackage{amscd}
\usepackage{amsmath}
\usepackage{amssymb}

\makeatletter
\@namedef{subjclassname@2010}{%
  \textup{2010} Mathematics Subject Classification}
\makeatother

\theoremstyle{plain}
\newtheorem*{acknowledgement}{Acknowledgement}

\newtheorem{corollary}{\bf Corollary}
\newtheorem{definition}{\bf Definition}
\newtheorem{lemma}{\bf Lemma}

\newtheorem{theorem}{\bf Theorem}

\theoremstyle{definition}

\numberwithin{equation}{section}

\begin{document}

\title[Critical metrics of the volume functional]{Critical Metrics of the Volume Functional on Manifolds with Boundary}

\author{H. Baltazar\,\,\,\,\&\,\,\,\,E. Ribeiro Jr}

\address[H. Baltazar]{Departamento de Matem\'{a}tica, Universidade Federal do Piau\'{\i}\\
64049-550 Te\-re\-si\-na, Piau\'{\i}, Brazil.}
\email{halyson@ufpi.edu.br}

\address[E. Ribeiro Jr]{Universidade Federal do Cear\'a - UFC, Departamento  de Matem\'atica, Campus do Pici, Av. Humberto Monte, Bloco 914,
60455-760, Fortaleza / CE, Brazil.}
\email{ernani@mat.ufc.br}
\thanks{E. Ribeiro and H. Baltazar were partially supported by CNPq/Brazil}
\date{November 5, 2015}

\begin{abstract}
The goal of this article is to study the space of smooth Riemannian structures on compact manifolds with boundary that satisfies a critical point equation associated with a boundary value problem. We provide an integral formula which enables us to show that if a critical metric of the volume functional on a connected $n$-dimensional manifold $M^n$ with boundary $\partial M$ has parallel Ricci tensor, then $M^n$ is isometric to a geodesic ball in a simply connected space form $\mathbb{R}^{n}$, $\mathbb{H}^{n}$ or $\mathbb{S}^{n}$.
\end{abstract}

\maketitle

\section{Introduction}
\label{intro}
An outstanding problem in differential geometry is to find Riemannian metrics on a given manifold $M^n$ that provides constant curvature. In this sense, it is crucial to understand the critical metrics of the Riemannian functionals, as for instance, the total scalar curvature functional and the volume functional. Einstein and Hilbert have proven that the critical points of the total scalar curvature functional restricted to the set of smooth Riemannian structures on $M^n$ of unitary volume are Einstein (cf. Theorem 4.21 in \cite{besse}). Moreover, the total scalar curvature functional restricted to a given conformal class is just the Yamabe functional, whose critical points are constant scalar curvature metrics in that class. Hilbert \cite{hilbert} proved that the equations of general relativity can be recovered from the total scalar curvature functional. From this, we have a natural way to prove the existence of Einstein metrics.
 
Inspired by a result obtained in \cite{fan} as well as in the characterization of the critical points of the total scalar curvature functional, Miao and Tam (cf. \cite{miaotam} and \cite{miaotamTAMS}) studied variational properties of the volume functional constrained to the space of metrics of constant scalar curvature on a given compact manifold with boundary. While Corvino, Eichmair and Miao \cite{CEM} studied the modified problem of finding stationary points for the volume functional on the space of metrics whose scalar curvature is equal to a given constant.

Following the terminology used in \cite{BDR} we recall the definition of Miao-tam critical metrics. 

\begin{definition}
\label{def1} A Miao-Tam critical metric is a 3-tuple $(M^n,\,g,\,f),$ where $(M^{n},\,g)$ is a compact Riemannian manifold of dimension at least three with a smooth boundary $\partial M$ and $f: M^{n}\to \Bbb{R}$ is a smooth function such that $f^{-1}(0)=\partial M$ satisfying the overdetermined-elliptic system
\begin{equation}
\label{eqMiaoTam} \mathfrak{L}_{g}^{*}(f)=g.
\end{equation} Here, $\mathfrak{L}_{g}^{*}$ is the formal $L^{2}$-adjoint of the linearization of the scalar curvature operator $\mathfrak{L}_{g}$. Such a function $f$ is called a potential function.
\end{definition}

We also recall that $$\mathfrak{L}_{g}^{*}(f)=-(\Delta f)g+Hess f-fRic,$$ where $Ric,$ $\Delta$ and $Hess $ stand, respectively, for the Ricci tensor, the Laplacian operator and the Hessian form on $M^n;$ see for instance \cite{besse}. Therefore, the fundamental equation of Miao-Tam critical metrics (\ref{eqMiaoTam}) can be rewritten as
\begin{equation}\label{eqMiaoTam1}
-(\Delta f)g+Hess f-fRic=g.
\end{equation}

In \cite{miaotam}, it was remarked that Miao-Tam critical metrics arise as critical points of the volume functional on $M^n$ when restricted to the class of metrics $g$ with prescribed constant scalar curvature such that $g_{|_{T \partial M}}=h$ for a prescribed Riemannian metric $h$  on the boundary. In addition, they showed that such metrics have constant scalar curvature $R.$ For more background see e.g. Proposition 2.1 and Theorem 2.3 in \cite{CEM}. Some explicit examples of Miao-Tam critical metrics can be found in  \cite{miaotam,miaotamTAMS}. They include the spatial Schwarzschild metrics and AdS-Schwarzschild metrics restricted to certain domains containing their horizon and bounded by two spherically symmetric spheres. We also remember that the standard metrics on geodesic balls in space forms are Miao-Tam critical metrics. For more details see Theorem 6 in \cite{miaotam}.

Miao and Tam \cite{miaotamTAMS} posed the question of whether there exist non-constant sectional curvature Miao-Tam critical metrics on a compact manifold whose boundary is isometric to a standard round sphere. In this sense, inspired by ideas outlined by Kobayashi \cite{kobayashi}, Kobayashi and Obata \cite{obata}, they proved that a locally conformally flat simply connected, compact Miao-Tam critical metric $(M^{n},g,f)$ with boundary isometric to a standard sphere $\mathbb{S}^{n-1}$ must be isometric to a geodesic ball in a simply connected space form $\mathbb{R}^{n}$, $\mathbb{H}^{n}$ or $\mathbb{S}^{n}$. 

In order to proceed, we recall three special tensors in the study of curvature for a Riemannian manifold $(M^n,\,g),\,n\ge 3.$  The first one is the Weyl tensor $W$ which is defined by the following decomposition formula
\begin{eqnarray}
\label{weyl}
R_{ijkl}&=&W_{ijkl}+\frac{1}{n-2}\big(R_{ik}g_{jl}+R_{jl}g_{ik}-R_{il}g_{jk}-R_{jk}g_{il}\big) \nonumber\\
 &&-\frac{R}{(n-1)(n-2)}\big(g_{jl}g_{ik}-g_{il}g_{jk}\big),
\end{eqnarray}
where $R_{ijkl}$ stands for the Riemann curvature tensor, whereas the second one is the Cotton tensor $C$ given by
\begin{equation}
\label{cotton} \displaystyle{C_{ijk}=\nabla_{i}R_{jk}-\nabla_{j}R_{ik}-\frac{1}{2(n-1)}\big(\nabla_{i}R
g_{jk}-\nabla_{j}R g_{ik}).}
\end{equation} Note that $C_{ijk}$ is skew-symmetric in the first two indices and trace-free in any two indices. These two above tensors are related as follows
\begin{equation}
\label{cottonwyel} \displaystyle{C_{ijk}=-\frac{(n-2)}{(n-3)}\nabla_{l}W_{ijkl},}
\end{equation}provided $n\ge 4.$ We also recall that the Bach tensor \cite{bach} on a Riemannian manifold $(M^n,g)$, $n\geq 4,$ is defined in term of the components of the Weyl tensor $W_{ikjl}$ as follows
\begin{equation}
\label{bach} B_{ij}=\frac{1}{n-3}\nabla^{k}\nabla^{l}W_{ikjl}+\frac{1}{n-2}R_{kl}W_{i}\,^{k}\,_{j}\,^{l},
\end{equation}
while for $n=3$ it is given by
\begin{equation}
\label{bach3} B_{ij}=\nabla^kC_{kij}.
\end{equation} We say that $(M^n,g)$ is Bach-flat when $B_{ij}=0.$ It is easy to check that locally conformally flat metrics as well as Einstein metrics are Bach-flat.

Recently, Barros, Di\'{o}genes and Ribeiro \cite{BDR}, based on the techniques outlined in a work of Cao and Chen \cite{CaoChen}, proved that a Bach-flat simply connected, compact Miao-Tam critical metric with boundary isometric to a standard sphere $\mathbb{S}^{3}$ must be isometric to a geodesic ball in a simply connected space form $\mathbb{R}^{4}$, $\mathbb{H}^{4}$ or $\mathbb{S}^{4}.$ Further, they showed that in dimension three the result even is true replacing the Bach-flat condition by the weaker assumption that $M^3$ has divergence-free Bach tensor. For more details, we refer the reader to \cite{BDR}.

At the same time, Miao and Tam  \cite{miaotamTAMS} studied these critical metrics under Einstein condition. In that case, they were able to remove the condition of boundary isometric to a standard sphere. More precisely, they obtained the following result.

\begin{theorem}[Miao-Tam, \cite{miaotamTAMS}]
\label{thmMT}
Let $(M^{n},g,f)$ be a connected, compact Einstein Miao-Tam critical metric with smooth boundary $\partial M.$ Then $(M^{n},g)$ is isometric to a geodesic ball in a simply connected space form $\mathbb{R}^{n}$, $\mathbb{H}^{n}$ or $\mathbb{S}^{n}.$
\end{theorem}

For what follows, we point out that every Riemannian manifold with parallel Ricci tensor has harmonic curvature. But, the converse statement is not true, see \cite{derd} and \cite{derd2} for more details. Indeed, there are examples of compact and noncompact Riemannian manifolds with harmonic curvature but non-parallel Ricci tensor. Here, motivated by the historical development on the study of critical metrics of the volume functional, we shall replace the assumption of Einstein in the Miao-Tam result (cf. Theorem \ref{thmMT}) by the parallel Ricci tensor condition, which is weaker that the former. In order to do so, we have established the following result.

\begin{theorem}\label{thmMain}
Let $(M^{n},g,f)$ be a compact, oriented, connected Miao-Tam critical metric with smooth boundary $\partial M.$ Then we have:
\begin{equation*}
\int_{M} f|divRm|^{2}dM_{g}+\frac{n}{n-1}\int_{M}|Ric-\frac{R}{n}g|^{2}dM_{g}+\int_{M}f(\Delta|Ric|^{2}-|\nabla Ric|^{2})dM_{g}=0.
\end{equation*}
\end{theorem}

In the sequel, as an application of Theorem \ref{thmMain} we get the following rigidity result.

\begin{corollary}\label{corA}
Let $(M^n,\,g,\,f)$ be a compact, oriented, connected Miao-Tam critical metric with parallel Ricci tensor and smooth boundary $\partial M.$ Then $(M^n,\,g)$ is isometric to a geodesic ball in a simply connected space form $\Bbb{R}^{n},$ $\Bbb{H}^{n},$ or $\Bbb{S}^n.$
\end{corollary}

It is easy to check that Einstein manifolds $M^n,$ $n\geq 3,$ have parallel Ricci tensor. Therefore, Corollary \ref{corA} clearly improves Theorem \ref{thmMT}. Moreover, it is worth pointing out that our arguments designed for the proof of Theorem \ref{thmMain} differ significantly from \cite{miaotamTAMS}.

Let us also highlight that every Einstein manifold has harmonic Weyl tensor (cf. \cite{derd}). Furthermore, we already know that there is no relationship between Bach-flat condition, considered in \cite{BDR}, and the condition that $M^n$ has harmonic Weyl tensor. Therefore, it is natural to ask which geometric implications has the assumption of the harmonicity of the Weyl tensor on a Miao-Tam critical metric. By using once more Theorem \ref{thmMain} jointly with Theorem \ref{thmMT}, it is straightforward to check that if a compact, oriented, connected Miao-Tam critical metric with harmonic Weyl tensor and smooth boundary satisfies 

\begin{equation}
\label{plk}
\int_{M} f\Delta|Ric|^{2}dM_{g}\geq\int_{M}f|\nabla Ric|^{2}dM_{g},
\end{equation} then $(M^n,\,g)$ is isometric to a geodesic ball in a simply connected space form $\Bbb{R}^{n},$ $\Bbb{H}^{n}$ or $\Bbb{S}^n.$ Nonetheless, it is very interesting to prove that condition (\ref{plk}) can be removed.

\section{Preliminaries}
\label{Preliminaries}

In this section we shall present some preliminaries which will be useful for the
establishment of the desired results. Firstly, we recall that the fundamental equation of a Miao-Tam critical metric (\ref{eqMiaoTam}) becomes
\begin{equation}\label{eq:miaotam}
-(\Delta f)g+Hessf-fRic=g.
\end{equation} In particular, tracing (\ref{eq:miaotam}) we have
\begin{equation}\label{eqtrace}
(n-1)\Delta f+Rf=-n.
\end{equation} For sake of simplicity, we now rewrite equation (\ref{eq:miaotam}) in the tensorial language as follows
\begin{equation}\label{eq:tensorial}
-(\Delta f)g_{ij}+\nabla_{i}\nabla_{j}f-fR_{ij}=g_{ij}.
\end{equation} Moreover, by using (\ref{eqtrace}) it is not difficult to check that
\begin{equation}
\label{IdRicHess} f\mathring{Ric}=\mathring{Hess f},
\end{equation} where $\mathring{T}$ stands for the traceless of $T.$

For our purpose we also remember that as consequence of Bianchi identity we have
\begin{equation}\label{Bianchi}
(div Rm)_{jkl}=\nabla_kR_{jl}-\nabla_lR_{jk}.
\end{equation}

Under these notations we get the following lemma obtained previously in \cite{BDR}.

\begin{lemma}[\cite{BDR}]
\label{L1}
Let $(M^{n},g,f)$ be a Miao-Tam critial metric. Then we have:

$$f(\nabla_{i}R_{jk}-\nabla_{j}R_{ik})=R_{ijkl}\nabla_{l}f+\frac{R}{n-1}(\nabla_{i}fg_{jk}-\nabla_{j}fg_{ik})-(\nabla_{i}fR_{jk}-\nabla_{j}f R_{ik}).$$
\end{lemma}
\begin{proof} Since the proof of this lemma is very short, we include it here for the sake of completeness. Firstly, we use  (\ref{eq:tensorial}) to infer

\begin{equation}
\label{eq1lem1}
(\nabla_{i}f)R_{jk}+f\nabla_{i}R_{jk}=\nabla_{i}\nabla_{j}\nabla_{k}f-(\nabla_{i}\Delta f)g_{jk}.
\end{equation} Next, since $M^n$ has constant scalar curvature we have from (\ref{eqtrace}) that $$\nabla_{i}\Delta f=-\frac{R}{n-1}\nabla_{i}f,$$ which substituted into (\ref{eq1lem1}) gives 

\begin{equation}
\label{eq2lem2}
f\nabla_{i}R_{jk}=-(\nabla_{i}f)R_{jk}+\nabla_{i}\nabla_{j}\nabla_{k}f+\frac{R}{n-1}\nabla_{i}fg_{jk}.
\end{equation} Finally, it suffices to apply the Ricci identity to arrive at
\begin{eqnarray*}
f(\nabla_{i}R_{jk}-\nabla_{j}R_{ik})=R_{ijkl}\nabla_{l}f+\frac{R}{n-1}(\nabla_{i}f g_{jk}-\nabla_{j}f g_{ik})-(\nabla_{i}f R_{jk}-\nabla_{j}f R_{ik}),
\end{eqnarray*} as we wanted to prove.
\end{proof}

To conclude this section we recall that, from commutation formulas (Ricci
identities), for any Riemannian manifold $M^n$ we have
\begin{equation}\label{idRicci}
\nabla_i\nabla_j R_{ik}-\nabla_j\nabla_i R_{ik}=R_{ijis}R_{sk}+R_{ijks}R_{is},
\end{equation} for more details see \cite{chow}.

\section{Proof of the Main Result}

In this section we shall prove Theorem \ref{thmMain} announced in Section \ref{intro}. First of all, we shall present a fundamental integral formula.

\begin{lemma}\label{L2}
Let $(M^{n},g,f)$ be a Miao-Tam critial metric. Then we have:
\begin{eqnarray*}
\int_{M} f|divRm|^{2}dM_{g}&=&2\int_{M} f|\nabla Ric|^{2}dM_{g}+\int_{M}\langle\nabla f, \nabla|Ric|^{2}\rangle dM_{g}\\&&
+2\int_{M}f (R_{ij}R_{ik}R_{jk}-R_{ijkl}R_{jl}R_{ik})dM_{g}\\&&+2\int_{M}C_{klj}\nabla_{l}f R_{jk}dM_{g}.
\end{eqnarray*}
\end{lemma}
\begin{proof} Firstly, we use (\ref{Bianchi}) to infer 
\begin{eqnarray*}
\int_{M} f|divRm|^{2}dM_{g}&=&\int_{M} f(divRm)_{jkl}(divRm)_{jkl} dM_{g}\\
&=&\int_{M} f(\nabla_{k} R_{lj}-\nabla_{l}R_{kj})(\nabla_{k} R_{lj}-\nabla_{l}R_{kj})dM_{g},
\end{eqnarray*} which can be rewritten as 

\begin{eqnarray}
\label{n1}
\int_{M} f|divRm|^{2}dM_{g}&=&2\int_{M} f|\nabla Ric|^{2}dM_{g}-2\int_{M} f\nabla_{k}R_{lj}\nabla_{l}R_{kj}dM_{g}.
\end{eqnarray} Next, notice that

\begin{eqnarray*}
\int_{M} f\nabla_{k}R_{lj}\nabla_{l}R_{kj}dM_{g}&=&-\int_{M}\nabla_{l}f\nabla_{k}R_{lj}R_{kj}dM_{g}-\int_{M} f\nabla_{l}\nabla_{k}R_{lj}R_{kj}dM_{g}\\&&+\int_{M}\nabla_{l}(f\nabla_{k}R_{lj}R_{kj})dM_{g}.
\end{eqnarray*} Hence we use this data in (\ref{n1}) to deduce

\begin{eqnarray}
\label{n2}
\int_{M} f|divRm|^{2}dM_{g}&=&2\int_{M} f|\nabla Ric|^{2}dM_{g}+2\int_{M}\nabla_{l}f\nabla_{k}R_{lj}R_{kj}dM_{g}\nonumber\\&&+2\int_{M} f\nabla_{l}\nabla_{k}R_{lj}R_{kj}dM_{g}-2\int_{M}\nabla_{l}(f\nabla_{k}R_{lj}R_{kj})dM_{g}.
\end{eqnarray} Since $M^n$ has constant scalar curvature we may use (\ref{cotton}) to infer

$$C_{klj}=\nabla_{k}R_{lj}-\nabla_{l}R_{kj},$$ which substituted jointly with (\ref{idRicci}) into (\ref{n2}) gives

\begin{eqnarray*}
\int_{M} f|divRm|^{2}dM_{g} &=&2\int_{M} f|\nabla Ric|^{2}dM_{g}+2\int_{M} \nabla_{l}f(\nabla_{l}R_{kj}+C_{klj})R_{kj}dM_{g}\\
&&+2\int_{M} f(\nabla_{k}\nabla_{l}R_{lj}+R_{lkls}R_{sj}+R_{lkjs}R_{ls})R_{kj}dM_{g}\\
&&-2\int_{M}\nabla_{l}(f\nabla_{k}R_{lj}R_{kj})dM_{g}.
\end{eqnarray*}

Now, since $f$ vanishes on the boundary, we apply Stokes's formula and the twice contracted second Bianchi identity to arrive at

\begin{eqnarray*}
\int_{M} f|divRm|^{2}dM_{g} &=&2\int_{M} f|\nabla Ric|^{2}dM_{g}+\int_{M}\langle\nabla f, \nabla|Ric|^{2}\rangle dM_{g}\\&&+2\int_{M} \nabla_{l}f C_{klj}R_{kj}dM_{g}\\
&&+2\int_{M}f (R_{ks}R_{sj}R_{kj}-R_{kljs}R_{ls}R_{kj})dM_{g}.
\end{eqnarray*}

Finally, it suffices to change the indices of the last integral of the right hand side to obtain

\begin{eqnarray*}
\int_{M} f|divRm|^{2}dM_{g} &=&2\int_{M} f|\nabla Ric|^{2}dM_{g}+\int_{M}\langle\nabla f, \nabla|Ric|^{2}\rangle dM_{g}\\&&+2\int_{M}C_{klj}\nabla_{l}f R_{jk}dM_{g}\\&&+2\int_{M}f (R_{ij}R_{ik}R_{jk}-R_{ijkl}R_{jl}R_{ik})dM_{g}.
\end{eqnarray*} So, the proof is completed.

\end{proof}

Before presenting the proof of Theorem \ref{thmMain}, we provide another expression for $\int_{M} f | divRm |^{2}dM_{g}.$ More precisely, we have established the following lemma.

\begin{lemma}\label{L3}
Let $(M^{n},g,f)$ be a Miao-Tam critial metric. Then we have:
\begin{eqnarray*}
\int_{M} f|divRm|^{2}dM_{g}&=&-\int_{M}\langle\nabla f, \nabla|Ric|^{2}\rangle dM_{g}-2\int_{M}C_{klj}\nabla_{l}f R_{jk}dM_{g}\\
&&-2\int_{M}f (R_{ij}R_{ik}R_{jk}-R_{ijkl}R_{jl}R_{ik})dM_{g}\\
&&+2\int_{M}\nabla_{i}(\nabla_{j}R_{ik}R_{jk})dM_{g}-2\int_{M}\nabla_{i}(R_{ijkl}R_{jl}\nabla_{k}f)dM_{g}.
\end{eqnarray*}
\end{lemma}

\begin{proof}
To begin with, from (\ref{eq:tensorial}) we achieve 
\begin{eqnarray*}
\int_{M}f R_{ijkl}R_{jl}R_{ik}dM_{g}&=&\int_{M}R_{ijkl}R_{jl}(-(\Delta f+1)g_{ik}+\nabla_{i}\nabla_{k}f)dM_{g}\\
&=&-\int_{M}(\Delta f+1)|Ric|^{2}dM_{g}+\int_{M}R_{ijkl}R_{jl}\nabla_{i}\nabla_{k}f dM_{g}.
\end{eqnarray*} From here it follows that

\begin{eqnarray*}
\int_{M}f R_{ijkl}R_{jl}R_{ik}dM_{g}&=&-\int_{M}(\Delta f+1)|Ric|^{2}dM_{g}-\int_{M}\nabla_{i}R_{ijkl}R_{jl}\nabla_{k}f dM_{g}\\
&&-\int_{M}R_{ijkl}\nabla_{i}R_{jl}\nabla_{k}f dM_{g}+\int_{M}\nabla_{i}(R_{ijkl}R_{jl}\nabla_{k}f)dM_{g}.
\end{eqnarray*}
By using (\ref{Bianchi}) we have

\begin{eqnarray}
\label{eqa}
\int_{M}f R_{ijkl}R_{jl}R_{ik}dM_{g}&=&-\int_{M}(\Delta f+1)|Ric|^{2}dM_{g}-\int_{M}C_{klj}R_{jl}\nabla_{k}f dM_{g}\nonumber\\
&&-\int_{M}R_{ijkl}\nabla_{i}R_{jl}\nabla_{k}f dM_{g}+\int_{M}\nabla_{i}(R_{ijkl}R_{jl}\nabla_{k}f)dM_{g}.
\end{eqnarray}

On the other hand, it is not difficult to check that 
\begin{eqnarray*}
\int_{M}R_{ijkl}\nabla_{i}R_{jl}\nabla_{k}fdM_{g}&=&\frac{1}{2}\Big[\int_{M}R_{ijkl}\nabla_{i}R_{jl}\nabla_{k}fdM_{g}+\int_{M}R_{ijkl}\nabla_{i}R_{jl}\nabla_{k}f dM_{g}\Big]\nonumber\\&=&\frac{1}{2}\Big[\int_{M}R_{ijkl}\nabla_{i}R_{jl}\nabla_{k}fdM_{g}+\int_{M}R_{jikl}\nabla_{j}R_{il}\nabla_{k}f dM_{g}\Big]\nonumber\\&=&\frac{1}{2}\Big[\int_{M}R_{ijkl}\nabla_{i}R_{jl}\nabla_{k}fdM_{g}-\int_{M}R_{ijkl}\nabla_{j}R_{il}\nabla_{k}f dM_{g}\Big]\nonumber\\&=&\frac{1}{2}\int_{M}R_{ijkl}\big(\nabla_{i}R_{jl}-\nabla_{j}R_{il}\big)\nabla_{k}f dM_{g}.
\end{eqnarray*} Then, we change the indices $k$ by $l$ in the right hand side to infer

$$\int_{M}R_{ijkl}\nabla_{i}R_{jl}\nabla_{k}fdM_{g}=-\frac{1}{2}\int_{M}R_{ijkl}\big(\nabla_{i}R_{jk}-\nabla_{j}R_{ik}\big)\nabla_{l}f dM_{g}.$$ This data substituted in (\ref{eqa}) yields

\begin{eqnarray}
\label{eq67}
\int_{M}f R_{ijkl}R_{jl}R_{ik}dM_{g}&=&-\int_{M}(\Delta f+1)|Ric|^{2}dM_{g}-\int_{M}C_{klj}R_{jl}\nabla_{k}f dM_{g}\nonumber\\
&&+\frac{1}{2}\int_{M}R_{ijkl}(\nabla_{i}R_{jk}-\nabla_{j}R_{ik})\nabla_{l}f dM_{g}\nonumber\\&&+\int_{M}\nabla_{i}(R_{ijkl}R_{jl}\nabla_{k}f)dM_{g}.
\end{eqnarray}

Next, it follows from Lemma \ref{L1} that
\begin{eqnarray*}
R_{ijkl}\nabla_{l}f(\nabla_{i}R_{jk}-\nabla_{j}R_{ik})&=&f|\nabla_{i}R_{jk}-\nabla_{j}R_{ik}|^{2}+(\nabla_{i}R_{jk}-\nabla_{j}R_{ik})(\nabla_{i}fR_{jk}-\nabla_{j}f R_{ik})\\&&-\frac{R}{n-1}(\nabla_{i}fg_{jk}-\nabla_{j}fg_{ik})(\nabla_{i}R_{jk}-\nabla_{j}R_{ik}).
\end{eqnarray*} Then, since $M^n$ has constant scalar curvature we can use the twice contracted second Bianchi identity to obtain

\begin{eqnarray*}
R_{ijkl}\nabla_{l}f(\nabla_{i}R_{jk}-\nabla_{j}R_{ik})&=&f|\nabla_{i}R_{jk}-\nabla_{j}R_{ik}|^{2}\nonumber\\&&+(\nabla_{i}R_{jk}-\nabla_{j}R_{ik})(\nabla_{i}fR_{jk}-\nabla_{j}f R_{ik}).
\end{eqnarray*} From this it follows that

\begin{eqnarray}
\label{lk1}
R_{ijkl}\nabla_{l}f(\nabla_{i}R_{jk}-\nabla_{j}R_{ik})&=&f|divRm|^{2}+\langle\nabla f,\nabla|Ric|^{2}\rangle-\nabla_{j}f R_{ik}\nabla_{i}R_{jk}\nonumber\\&&-\nabla_{i}f R_{jk}\nabla_{j}R_{ik}\nonumber\\
&=&f|divRm|^{2}+\langle\nabla f,\nabla|Ric|^{2}\rangle-2\nabla_{i}f R_{jk}\nabla_{j}R_{ik}.
\end{eqnarray} By using once more the twice contracted second Bianchi identity we immediately have
$$\nabla_{i}(\nabla_{j}f R_{ik}R_{jk})=\nabla_{i}\nabla_{j}f R_{ik}R_{jk}+\nabla_{i}f R_{jk}\nabla_{j}R_{ik}.$$ This combined with (\ref{lk1}) yields

\begin{eqnarray*}
R_{ijkl}\nabla_{l}f(\nabla_{i}R_{jk}-\nabla_{j}R_{ik})&=&f|divRm|^{2}+\langle\nabla f,\nabla|Ric|^{2}\rangle+2\nabla_{i}\nabla_{j}f R_{ik}R_{jk}\\&&-2\nabla_{i}(\nabla_{j}f R_{ik}R_{jk})\\
&=&f|divRm|^{2}+\langle\nabla f,\nabla|Ric|^{2}\rangle+2(f R_{ij}+(\Delta f+1)g_{ij}) R_{ik}R_{jk}\\
&&-2\nabla_{i}(\nabla_{j}f R_{ik}R_{jk}).
\end{eqnarray*} From which follows that

\begin{eqnarray*}
R_{ijkl}\nabla_{l}f(\nabla_{i}R_{jk}-\nabla_{j}R_{ik}) &=&f|divRm|^{2}+\langle\nabla f,\nabla|Ric|^{2}\rangle+2f R_{ij}R_{ik}R_{jk}\\&&+2(\Delta f+1)|Ric|^{2}-2\nabla_{i}(\nabla_{j}f R_{ik}R_{jk}).
\end{eqnarray*}

Whence, on integrating this last expression over $M^n$ and then substituting into (\ref{eq67}) we arrive at

\begin{eqnarray*}
\int_{M}f R_{ijkl}R_{jl}R_{ik}dM_{g}&=&-\int_{M}C_{klj}R_{jl}\nabla_{k}f dM_{g}+\int_{M}\nabla_{i}(R_{ijkl}R_{jl}\nabla_{k}f)dM_{g}\\&&+\frac{1}{2}\int_{M}f|divRm|^{2}dM_{g}+\frac{1}{2}\int_{M}\langle\nabla f,\nabla|Ric|^{2}\rangle dM_{g}\\&&+\int_{M}f R_{ij}R_{ik}R_{jk}dM_{g}-\int_{M}\nabla_{i}(\nabla_{j}f R_{ik}R_{jk})dM_{g}.
\end{eqnarray*} Since the Cotton tensor $C$ is skew-symmetric in the first two indices we get
\begin{eqnarray*}
\int_{M} f|divRm|^{2}dM_{g}&=&-\int_{M}\langle\nabla f, \nabla|Ric|^{2}\rangle dM_{g}-2\int_{M}C_{klj}\nabla_{l}f R_{jk}dM_{g}\\
&&-2\int_{M}f (R_{ij}R_{ik}R_{jk}-R_{ijkl}R_{jl}R_{ik})dM_{g}\\
&&+2\int_{M}\nabla_{i}(\nabla_{j}f R_{ik}R_{jk})dM_{g}-2\int_{M}\nabla_{i}(R_{ijkl}R_{jl}\nabla_{k}f)dM_{g}.
\end{eqnarray*} This gives the requested result.
\end{proof}

We are now in the position to prove Theorem \ref{thmMain}.

\subsection{Proof of Theorem \ref{thmMain}}

\begin{proof} First of all, we sum the expressions obtained in Lemma \ref{L2} and Lemma \ref{L3} to infer

\begin{eqnarray*}
\int_{M} f|divRm|^{2}dM_{g}&=&\int_{M} f|\nabla Ric|^{2}dM_{g}+\int_{M}\nabla_{i}(\nabla_{j}f R_{ik}R_{jk})dM_{g}\\&&-\int_{M}\nabla_{i}(R_{ijkl}R_{jl}\nabla_{k}f)dM_{g}.
\end{eqnarray*} Now, we change the indices $k$ by $l$ in the third integral of the right hand side to get

\begin{eqnarray}
\label{kp}
\int_{M} f|divRm|^{2}dM_{g}&=&\int_{M} f|\nabla Ric|^{2}dM_{g}+\int_{M}\nabla_{i}(\nabla_{j}f R_{ik}R_{jk})dM_{g}\nonumber\\&&+\int_{M}\nabla_{i}(R_{ijkl}\nabla_{l}f R_{jk})dM_{g}.
\end{eqnarray} Then, substituting Lemma \ref{L1} into (\ref{kp}) we deduce

\begin{eqnarray*}
\int_{M} f|divRm|^{2}dM_{g}&=&\int_{M} f|\nabla Ric|^{2}dM_{g}+\int_{M}\nabla_{i}(\nabla_{j}f R_{ik}R_{jk})dM_{g}\nonumber\\
&&+\int_{M}\nabla_{i}\Big[f C_{ijk}R_{jk}-\frac{R}{n-1}(\nabla_{i}f R-\nabla_{j}f R_{ji})\\&&+(\nabla_{i}f|Ric|^{2}-\nabla_{j}fR_{ik}R_{jk})\Big]dM_{g}.
\end{eqnarray*} Upon rearranging the terms we use Stokes's formula to arrive at

\begin{eqnarray*}
\int_{M} f|divRm|^{2}dM_{g} &=&\int_{M} f|\nabla Ric|^{2}dM_{g}\nonumber\\&&+\int_{M}\nabla_{i}\Big[-\frac{R^{2}}{n-1}\nabla_{i}f +\frac{R}{n-1} R_{ji}\nabla_{j}f+\nabla_{i}f|Ric|^{2}\Big]dM_{g}\nonumber\\
&=&\int_{M} f|\nabla Ric|^{2}dM_{g}\nonumber\\&&+\int_{M}\Big[-\frac{R^{2}}{n-1}\Delta f+\frac{R}{n-1}\nabla_{i}\nabla_{j}f R_{ji}+\Delta f|Ric|^{2}+\langle\nabla f,\nabla|Ric|^{2}\rangle\Big]dM_{g}.
\end{eqnarray*} Therefore, by using (\ref{eq:tensorial}) and (\ref{eqtrace}) we achieve

\begin{eqnarray*}
\int_{M} f|divRm|^{2}dM_{g} &=&\int_{M}\Big[-\frac{R^{2}}{n-1}\Delta f+\frac{R}{n-1}(f R_{ij}+(\Delta f+1)g_{ij})R_{ji}+\Delta f|Ric|^{2}\Big]dM_{g}\\
&&+\int_{M} f|\nabla Ric|^{2}dM_{g}+\int_{M}\langle\nabla f,\nabla|Ric|^{2}\rangle dM_{g}\\
&=&\int_{M}\Big[\frac{R}{n-1}f|Ric|^{2}+\frac{R^{2}}{n-1} +\frac{-Rf-n}{n-1}|Ric|^{2}\Big]dM_{g}\\
&&+\int_{M} f|\nabla Ric|^{2}dM_{g}+\int_{M}\langle\nabla f,\nabla|Ric|^{2}\rangle dM_{g}.
\end{eqnarray*} This implies that

\begin{eqnarray*}
\int_{M} f|divRm|^{2}dM_{g}&=&\int_{M} f|\nabla Ric|^{2}dM_{g}-\frac{n}{n-1}\int_{M}\Big(|Ric|^{2}-\frac{R^{2}}{n}\Big)dM_{g}\\&&+\int_{M}\langle\nabla f,\nabla|Ric|^{2}\rangle dM_{g}\\
&=&\int_{M} f|\nabla Ric|^{2}dM_{g}-\frac{n}{n-1}\int_{M}|Ric-\frac{R}{n}g|^{2}dM_{g}\\&&+\int_{M}\langle\nabla f,\nabla|Ric|^{2}\rangle dM_{g}\\
&=&\int_{M} f|\nabla Ric|^{2}dM_{g}-\frac{n}{n-1}\int_{M}|Ric-\frac{R}{n}g|^{2}dM_{g}\\&&-\int_{M}f\Delta|Ric|^{2}dM_{g}.
\end{eqnarray*} From which we deduce
$$\int_{M} f|divRm|^{2}dM_{g}+\frac{n}{n-1}\int_{M}|Ric-\frac{R}{n}g|^{2}dM_{g}+\int_{M}f(\Delta|Ric|^{2}-|\nabla Ric|^{2})dM_{g}=0.$$This what we wanted to prove. 
\end{proof}

\subsection{Proof of Corollary \ref{corA}}
\begin{proof} We already know that every Riemannian manifold with parallel Ricci tensor has harmonic curvature (cf. Equation (\ref{Bianchi})). Moreover, from  {\it Kato's inequality} we have 
$$|\nabla|Ric||\leq|\nabla Ric|.$$ From this, since $M^n$ has parallel Ricci tensor, we conclude that $|Ric|$ is constant on $M^n.$ Now, it suffices to invoke Theorem \ref{thmMain} to deduce $$\frac{n}{n-1}\int_{M}|Ric-\frac{R}{n}g|^{2}dM_{g}=0$$ and this forces $M^n$ to be Einstein. Then, we are in position to use Theorem \ref{thmMT} (see also Theorem 1.1 of  \cite{miaotamTAMS}) to conclude that $(M^n ,\,g)$ is isometric to a geodesic ball in a simply connected space form $\Bbb{R}^{n},$ $\Bbb{H}^{n}$ or $\Bbb{S}^{n}.$ So, the proof is completed.
\end{proof}

\begin{acknowledgement}
The authors want to thank A. Barros, R. Di\'ogenes  and P. Miao for helpful conversations about this subject. 
\end{acknowledgement}

\end{document}